\documentclass{amsart}

\usepackage{graphicx} 
\usepackage{amsmath}
\usepackage{amscd}
\usepackage{amssymb}

\theoremstyle{plain}
\newtheorem{theorem}{Theorem}

\newtheorem{lemma}{Lemma}
\newtheorem{fact}{Fact}

\theoremstyle{definition}
\newtheorem{definition}{Definition}

\theoremstyle{remark}
\newtheorem{remark}{Remark}

\newcommand{\rk}{\operatorname{rank}}
\newcommand{\calH}{\mathcal{H}^{i, j}}

\newcommand{\pc}{
\begin{picture}(30,30)
\put(10,20){\vector(-1,1){10}}
\put(0,0){\vector(1,1){30}}
\put(30,0){\line(-1,1){10}}
\end{picture}
}

\newcommand{\nc}{
\begin{picture}(30,30)
\put(20,20){\vector(1,1){10}}
\put(30,0){\vector(-1,1){30}}
\put(0,0){\line(1,1){10}}
\end{picture}
}

\newcommand{\pzero}{
\begin{picture}(30,30)
\put(0,30){\vector(-1,1){1}}
\qbezier(0,0)(15,15)(0,30)
\put(30,30){\vector(1,1){1}}
\qbezier(30,0)(15,15)(30,30)
\end{picture}
}
\newcommand{\unknot}{\operatorname{unknot}}

\begin{document}
\title[On a Poincar\'{e} polynomial from Khovanov homology]{On a Poincar\'{e} polynomial from Khovanov homology and Vassiliev invariants}
\author{Noboru Ito}
\address{Graduate School of Mathematical Sciences, The University of Tokyo, 3-8-1, Komaba, Meguro-ku, Tokyo 153-8914, Japan}
\email{noboru@ms.u-tokyo.ac.jp}
\author{Masaya Kameyama}
\address{Graduate School of Mathematics, Nagoya University, Nagoya, 464-8602, Japan}
\email{m13020v@math.nagoya-u.ac.jp}
\date{\today}
\maketitle

\begin{abstract}
We introduce a Poincar\'{e} polynomial with two-variable $t$ and $x$ for knots, derived from Khovanov homology, where the specialization $(t, x)$ $=$ $(1, -1)$ is a Vassiliev invariant of order $n$.  Since for every $n$, there exist non-trivial knots with the same value of the Vassiliev invariant of order $n$ as that of the unknot, there has been no explicit formulation of a  perturbative knot invariant which is a coefficient of $y^n$ by the replacement $q=e^y$ for the quantum parameter $q$ of a quantum knot invariant, and which distinguishes the above knots together with the unknot.  The first formulation is our polynomial.       
\end{abstract}

\section{Introduction}\label{intro}
Vassiliev \cite{vassiliev} introduces his ordered invariants by using singularity theory.  For the space $\mathcal{M}$ of all smooth maps from $S^1$ to $\mathbb{R}^3$, let $\Sigma$ be the set of maps which are not embeddings.  Then, a filtration of subgroups $\{ G_n \}_{n=1}^{\infty}$ of the reduced cohomology $\tilde{H}^0 (\mathcal{M} \setminus \Gamma)$ is introduced.   
An element in $G_n \setminus G_{n-1}$ corresponds to an oriented knot $K$ gives us a knot invariant, which is so-called a {\it Vassiliev invariant} of order $n$.   
Birman and Lin \cite{BL} give a relation between the \emph{Jones polynomial} and the Vassiliev invariant, i.e., for a one-variable polynomial $U_x (K)$ obtained from the Jones polynomial by replacing the variable with $e^x$, they  show that for a power series $
U_x (K) = \sum_{n=0}^{\infty} u_n (K) x^n 
$, each $u_n$ is a Vassiliev invariant  of order $n$ (Fact~\ref{BLthm}).  

In this paper, we consider an analogue of this Birman-Lin argument using  Khovanov homology as follows.   
For an oriented link $L$, Khovanov \cite{Khovanov} defines groups that are knot invariants and are so-called {\it Khovanov homology} $\mathcal{H}^{i, j}(L)$ such that 
\[
J(L)(q) = \sum_{i, j} (-1)^i q^j \, \rk\, \calH (L), 
\]
where $J(L)(q)$ is a version of the Jones polynomial of $L$.  It implies  the \emph{Khovanov polynomial} 
\[
\sum_{i, j} t^i q^j \, \rk\, \calH (L) 
\quad (=Kh(L)(t, q)). \]
Using each coefficient $v_{n} (K) (t, x)$ of $y^n$ in $Kh(L)(t, q)|_{q= x e^y}$, we have:    

\begin{theorem}\label{main}
Let $l$, $m$, and $n$ be integers where $2 \le l < m < n$.  
Let $v_n (K) (t, x)$ be a function as in Definition~\ref{one-variable}.  Then,    
$v_n (K)(-1, 1)$ is a Vassiliev invariant of order $n$ and there exists a set $\{ K_{\mu} \}_{\mu \ge 2}$ consisting of oriented knots such that for a given tuple $(l, m, n)$, $v_n (K_l)(-1, 1)$ $=$ $v_n (K_m)(-1, 1)$ $= v_n (\unknot)(-1, 1)$ but $v_{n} (K_l) (t, x)$ $\neq$ $v_{n} (K_m) (t, x)$, $v_{n} (K_{l}) (t, x)$ $\neq$ $v_{n} (\unknot) (t, x)$, and $v_{n}  (K_{m}) (t, x)$ $\neq$ $v_{n} (\unknot) (t, x)$.  
\end{theorem}
\begin{remark}\label{remark1}
If $n=3$, $l=2$, there exists an oriented knot $K_2$ such that $v_3 (K_2)(-1, 1)$ $=$ $v_3 (\unknot)(-1, 1)$ wheres $v_{3} (K_2)(t, x)$ $\neq$ $v_{3} (\unknot) (t, x)$.  The proof is placed on the end of Section~\ref{Proof_main}.  
\end{remark}
\begin{remark}
This $v_{n}(K)(t, x)$ equals $\sum_{i, j} \frac{j^n}{n!} t^i \rk \mathcal{H}^{i, j} (K)\, x^j$ (Lemma~\ref{lemma1}), which implies a triply graded homology $\mathcal{H}_n^{i, j} (K)$ by assigning $n$ to $\mathcal{H}^{i, j} (K)$ that belongs to the coefficient of $y^n$, i.e., a formula   
\[
Kh(L)(t, xe^y) = \sum_{n=0}^{\infty} \sum_{i, j}    \frac{j^n}{n!} t^i \, \rk\, \calH_n (L)\,x^j  y^n,     
\]
satisfying $Kh(L)(-1, e^y)$ $=$ $J_L (q)$ holds 
(cf.~(\ref{align1}) of the proof of Lemma~\ref{lemma1}).  
\end{remark}

To the best our knowledge, there has been no explicit formulation of a \emph{perturbative}\footnote{The word ``perturbative" comes from Chern-Simons perturbation theory. A representative physical approach to Khovanov polynomial is refined Chern-Simons theory \cite{Aganagic-Shakirov}. However, perturbative calculations can not be applied for refined Chern-Simons theory. Therefore, we emphasize that the meaning of ``perturbative" in this paper is an analogy of Birman-Lin.} knot invariant which is a coefficient of $y^n$ obtained from the replacement $q=e^y$ of the quantum parameter $q$ of a quantum knot invariant, and which distinguishes $K_{m}$ ($m \le n-1$) of Theorem~\ref{main} (Figure~\ref{ntri}) together with the unknot wheres the Vassiliev invariant cannot.  The first formulation is our two-variable Poincar\'{e} polynomial $v_{n}(K_n)(t, x)$ which is  introduced in this paper, and which is the coefficient of $y^n$ and satisfies that the specialization $v_{n}(K_n)(-1, 1)$ is a Vassiliev invariant of order $n$.    Further, it is interesting that though this polynomial invariant $v_{n}(K)(t, x)$ can detect the difference between $K_{m}$ ($m \le n-1$) and the unknot, essentially, there exists a fixed number $j_0$ such that the coefficient $x^{j_0} y^n$ detect them (here, $j_0$ is actually the lowest degree of $x$ in $v_n (K) (t, x)$).  It implies that an information of the $j$-grade of $v_{n}(K)(t, x)$ is useful (for the detail, see Section~\ref{Proof_main}). In the literature,  this usefulness of the grade implicitly appeared in a work of Kanenobu-Miyazawa \cite{Kanenobu-Miyazawa}, they showed that $V^{(n)}_{K} (1)$ is a Vassiliev invariant by using the $n$th derivative of the Jones polynomial $V_K (r)$.         

The plan of the paper is as follows.  We will prove Theorem~\ref{main} (Section~\ref{Proof_main}) after we obtain definitions and notations (Section~\ref{sec_dfn}).  
In Section~\ref{SecTable}, we give a table of our function $v_{n, j} (K) (t, x)$ and its sum $v_{n} (K) (t, x)$.

\section{Preliminaries}\label{sec_dfn}
\subsection{The Jones polynomial and the Vassiliev invariant}
\begin{definition}[normalized Jones polynomial]\label{defJones}
Let $L$ be an oriented link.  
The Jones polynomial $V_L(r)$ is well-known, which is a polynomial in $\mathbb{Z}[r^{1/2}, r^{-1/2}]$ that is determined by an isotopy class of $L$.  The Jones polynomial $V_L(r)$ is defined by  
\begin{align*}
& V_{\text{unknot}}(r)=1, \\
& r^{-1} V_{L_+}(r)-r V_{L_-}(r)=(r^{1/2} - r^{-1/2}) V_{L_0}(r).  
\end{align*}
where links $L_+$, $L_-$, and $L_0$ are defined by Figure~\ref{skein1} and where Figure~\ref{skein1} corresponds to local figures are included on a neighborhood and the exteriors of the three neighborhoods are the same.  
\begin{figure}[htbp]
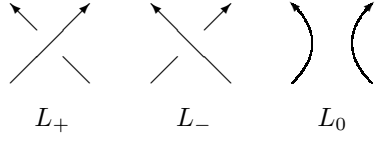

$\pc$ \qquad $\nc$ \qquad $\pzero$
\[ L_+ \qquad\qquad  L_- \qquad\qquad  L_0 \]
\caption{Local figures $L_+$, $L_-$, and $L_0$ of an oriented link.}\label{skein1}
\end{figure}
\end{definition}
\begin{definition}[unnormalized Jones polynomial]\label{verJones}
Letting $q= -r^{1/2}$, we define an unnormalized Jones polynomial $J_L (q)$ by
\[
J_L (q)|_{q= -r^{1/2}} = (-r^{1/2}-r^{-1/2}) V_L (r).  
\]
By definition, $J_L (q)$ is a polynomial in $\mathbb{Z}[q, q^{-1}]$ that is determined by an isotopy class of $L$.  Let $L_+$, $L_-$, and $L_0$ be as in Definition~\ref{defJones}.  Then,   
the polynomial $J_L (q)$ satisfies  
\begin{align*}
& J_{\text{unknot}}(q)= q + q^{-1}, \\
& q^{-2} J_{L_+}(q)- q^2 J_{L_-}(q)=(q^{-1} -q) J_{L_0}(q).  
\end{align*}
\end{definition}
\begin{fact}[Birman-Lin, Theorem of \cite{BL}]\label{BLthm}
Let $K$ be a knot and let $V_K (r)$ be its Jones polynomial as in Definition~\ref{defJones}.  Let $U_x (K)$ be obtained from $V_K (r)$ by replacing the variable $r$ by $e^x$.  Express $U_x (K)$ as a power series in $x$: 
\[
U_x (K) = \sum_{i=0}^{\infty} u_i (K) x^i.   
\] 
Then, $u_0 (K)=1$ and each $u_i (K)$, $i \ge 1$ is a Vassiliev invariant of order $i$.   
\end{fact}

\subsection{A polynomial invariant  from Khovanov polynomial}\label{s2}
\begin{definition}\label{not1}
Let $L$ be a link and $\calH (L)$ the Khovanov homology group of $L$.   The Khovanov polynomial is defined by 
\[
Kh(L)(t, q) = \sum_{i, j} t^i q^j \, \rk\, \calH (L).  
\]
\end{definition}
\begin{definition}[two-variable polynomials]\label{one-variable}
Let $Kh (K) (t, q) |_{q = x e^y}$ be a polynomial obtained from the Khovanov polynomial $Kh(K) (t, q)$ by replacing the variable $q$ with $x e^y$.  Then, let $v_n (K) (t, x)$ the coefficient of $y^n$ and let $v_{n, j} (K) (t, x)$ be (the coefficient of $x^j y^n$) $\cdot$ $x^j$.    
\end{definition}
By definition, $v_{n} (K) (t, x)$ $=\sum_j v_{n, j} (K) (t, x)$.  
It is clear that every $v_{n, j} (K) (t, x)$ is a link invariant, which implies that $v_{n} (K) (t, x)$ is also a link invariant.    
Definition~\ref{not1} and Definition~\ref{one-variable} imply Lemma~\ref{lemma1}.  
\begin{lemma}\label{lemma1}
\[
v_{n, j} (K)(t, x) = \frac{j^n}{n!} \sum_i t^i \rk \mathcal{H}^{i, j} (K)\, x^j.
\]
As a corollary, 
\[v_{n} (K)(t, x) = \sum_{i, j} \frac{j^n}{n!} t^i \rk \mathcal{H}^{i, j} (K)\, x^j.
\]
\end{lemma}
\begin{proof}
\begin{align*}
Kh(L)(t, x e^y) &= \sum_{j} (e^y)^j x^j \sum_i t^i \, \rk\, \calH (L) \\
&= \sum_{j} \sum_{n=0}^{\infty} \frac{(jy)^n}{n!} x^j \sum_i t^i \, \rk\, \calH (L) \\
&= \sum_{j} \sum_{n=0}^{\infty}  \frac{j^n}{n!}  \sum_i t^i \, \rk\, \calH (L)\,x^j  y^n.
\end{align*}
Then, the coefficient of $x^j y^n$ is $\frac{j^n}{n!}  \sum_i t^i \, \rk\, \calH (L)$.  
This fact together with Definition~\ref{one-variable} of $v_{n, j} (K) (t, x)$, we have 
\[
v_{n, j} (K) (t, x) = \frac{j^n}{n!}  \sum_i t^i \, \rk\, \calH (L)\, x^j.  
\]
As a corollary, 
\[v_{n} (K)(t, x) = \sum_j v_{n, j} (K)(t, x) = \sum_{i, j} \frac{j^n}{n!}  t^i \rk \mathcal{H}^{i, j} (K)\, x^j.
\]
\end{proof}
\begin{lemma}\label{lemma2}
The integer $v_{n} (K) (-1, 1)$ is 
a Vassiliev invariant of order $n$.  

As a corollary, every Vassiliev invariant of order $n$ has a presentation 
\[
v_{n} (K)(-1, 1) = \sum_j v_{n, j} (K)(-1, 1).  
\]
\end{lemma}
\begin{proof}
Using the above proof of Lemma~\ref{lemma1}, setting $x=1$ and $t=-1$, we have
\begin{align}\label{align1}
J(L)(q)|_{q={e^y}} = Kh(L)(-1, e^y) 
&= \sum_{j} \sum_{n=0}^{\infty}  \frac{j^n}{n!}  \sum_i (-1)^i \, \rk\, \calH (L)\,  y^n.
\end{align}
The coefficient of $y^n$ is $\sum_j \frac{j^n}{n!} \sum_i (-1)^i \, \rk\, \calH (L)$, which is $v_{n} (K) (-1, 1)$.  
Then, by the same argument as \cite[Proof of Theorem~4.1]{BL} of Birman-Lin, it is elementary to prove  that the coefficient of $y^n$ of $J(L)(q)|_{q={e^y}}$ is a Vassiliev invariant of order $n$.  This fact and Lemma~\ref{lemma1} imply the formula of the claim.  
\end{proof}

\section{A proof of Theorem~\ref{main}.}\label{Proof_main}
Since Lemma~\ref{lemma2} holds, we should the latter part of the claim.  For this proof, we use notations and definitions of Khovanov homology as in \cite{viro}.  Although it is sufficient to use $\mathbb{Z}/2 \mathbb{Z}$-homology, here we use $\mathbb{Z}$-homology to avoid adding notations of  symbols.  
We recall that a chain group $\mathcal{C}^{i, j}(D)$ of an oriented link diagram $D$.  In particular, for each enhanced state of $\mathcal{C}^{i, j}(D)$, $i(S)=\frac{w(D)-\sigma(s)}{2}$ and  $j(S)=w(D)$ $+$ $i(S)$ $+$ $\tau(S)$ (for definition of a state $s$, an enhanced state $S$, the writhe number $w(D)$, a sum $\sigma(s)$ of signs , and a sum $\tau(S)$ of signs, see \cite{viro}).  
   

Let $m$ be a positive integer $(m \ge 2)$ and $K_n$ a knot with a fixed $m$ that is defined by Figure~\ref{ntri}.  It is well-known that for every Vassiliev invariant $v_n$ of order $n$, $v_n (\unknot)=0$ and $v_n (K_m)=0$ ($m \le n-1$) \cite{ohyama}, which implies that $v_n (K_l)(-1, 1)$ $=$ $v_n (K_m) (-1, 1)$ $=$ $v_n (\unknot)(-1, 1)$ ($\because$ Lemma~\ref{lemma2}).  

Let $D_m$ be a knot diagram defined by Figure~\ref{ntri}, $s_n$ a state defined by Figure~\ref{nstateR}~(a), and $S_n$ a state defined by Figure~\ref{nstateR}~(b).    
\begin{figure}[h]
\includegraphics[width=5cm]{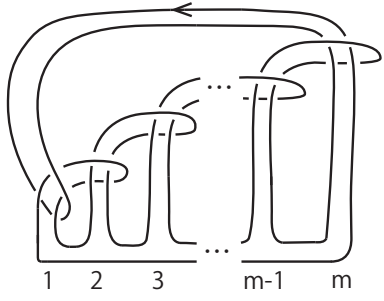}
\caption{A diagram of a knot $K_m$.}\label{ntri}
\end{figure}
\begin{figure}[h]
\includegraphics[width=12cm]{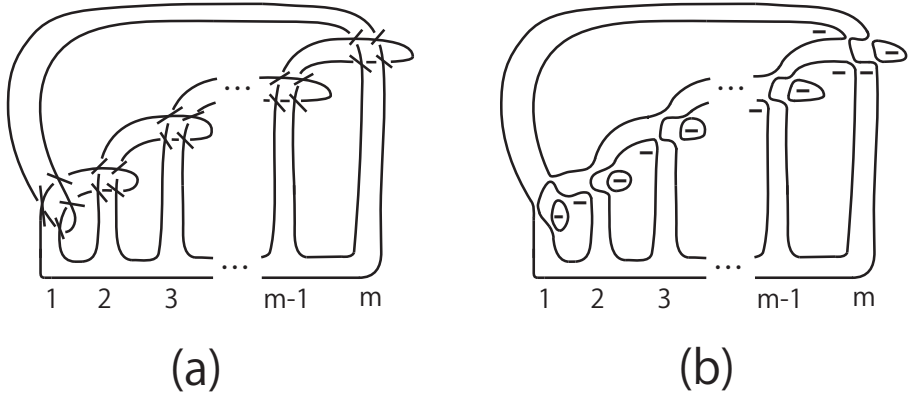}
\caption{(a) the state $s$ of $K_m$ (each short edge indicates the direction of smoothing of a crossing) and (b) the enhanced state $S$ of $K_m$ (each sign indicates a sign of a circle).}\label{nstateR}
\end{figure} 
  
Note that by the definition of this $\mathbb{Z}$-homology, $S_m$ obtains the minimum number of degree $i$ is $-2m$ and the minimum number of degree $j$ is $-4m-1$ as follows: 
\begin{align*}
&w(D_m) = 0, \\
&i(S_m) = \frac{0-4m}{2}= -2m,~{\textrm{and}}~\\
&j(S_m) = 0+(-2m)+(-1-2m) = -4m-1.
\end{align*}
Note also that by the definition of the differential $d : \mathcal{C}^{i, j}(D)$ $\to$ $\mathcal{C}^{i+1, j}(D)$, $d^{-2m} (S_m)$ $=$ $0$ and ${\operatorname{Im}} d^{-2m-1}=0$.  Then, for each $m$ ($\ge 2$), 
\begin{equation}\label{eq2n}
\mathcal{H}^{-2m, -4m-1}(K_m)= \mathbb{Z}.  
\end{equation}
By Lemma~\ref{lemma1}, 
\[
v_{n, j} (K)(t, x) = \sum_i t^i \rk \mathcal{H}^{i, j} (K)\, x^j.  
\]
We focus on the minimum number of $i$ that is $-2m$, and the minimum number of $j$ that is $-4m-1$.  
Setting $i=-2m$ and $j=-4m-1$, the coefficient of $t^{-2m} x^{-4m-1}$ in $v_{n, -4m-1}(K)(t, x)$ is  
\[
\frac{(-4m-1)^n}{n!} \rk \mathcal{H}^{-2m, -4m-1} (K).
\]
Then, (\ref{eq2n}) implies
\[
\frac{(-4m-1)^n}{n!} \rk \mathcal{H}^{-2m, -4m-1} (K_m) = \frac{(-4m-1)^n}{n!}.
\]

Thus, for every pair $l, m$ ($2 \le l < m$), $v_{n, -4l-1} (K_l) (t, x)$ $=$ $\frac{(-4l-1)^n}{n!}$ $\neq$ $\frac{(-4m-1)^n}{n!}$ $=$ $v_{n, -4m-1} (K_m) (t, x)$.  
Here, recall that for the unknot, it is well-known that $Kh(\unknot)(t, q)$ $=$ $q^{-1}$ $+$ $q$, which implies that there is no non-trivial coefficient of $t^k$ $(k \neq 0)$, i.e., any non-trivial part corresponds to the coefficient $q+q^{-1}$, which belongs to the coefficient of $t^0$.  It implies $v_{n, -4l-1} (K_l) (t, x)$ $\neq$ $v_{n, -4l-1} (\unknot) (t, x)$ and $v_{n, -4m-1} (K_m) (t, x)$ $\neq$ $v_{n, -4m-1} (\unknot) (t, x)$.  

Note that for every knot $K_l$ ($2 \le l$), the minimum number of $j$ is $-4l-1$, by always focusing on the lowest degree of $x$ in $v_n (K) (t, x)$ ($=$ $\sum_j v_{n, j} (K) (t, x)$), the above argument works since the coefficient of the lowest degree of $x$ exactly equals $v_{n, -4l-1}(K) (t, x)$.  
Therefore, by focusing the case $j=-4l-1$ or the case $j=-4m-1$, 
for every pair $l, m$ ($2 \le l < m$), $v_{n} (K_l) (t, x)$ $\neq$ $v_{n} (K_m) (t, x)$, $v_{n} (K_l) (t, x)$ $\neq$ $v_{n} (\unknot) (t, x)$ and $v_{n} (K_m) (t, x)$ $\neq$ $v_{n} (\unknot) (t, x)$ since for each case, two lowest degrees are different.
It completes the proof of Theorem~\ref{main}.   
$\hfill\Box$

\emph{Proof of Remark~\ref{remark1}}.  
Note that the coefficient of $t^{- 2 \cdot 2} x^{-4 \cdot 2 -1}$ in $v_{3, -9} (K_2)(t, x)$ is $\frac{(-4 \cdot 2 -1)^3}{3!}$.  Thus, $v_{3, -9} (K_2)(t, x)$ $\neq$ $v_{3, -9} (\unknot)(t, x)$.  By focusing on the lowest degree of $x$ in $v_3 (K) (t, x)$ ($=$ $\sum_j v_{3, j} (K) (t, x)$), we have the statement of Remark~\ref{remark1}.~$\hfill\Box$

\section{Table}\label{SecTable}
We give some examples of the Khovanov polynomial and the two-variables polynomials for a few prime knots.
We use the data of the Khovanov polynomial in the Mathematica package KnotTheory \cite{KnotAtlas} and attach a Mathematica file to arXiv page.

\begin{align*}
 \begin{array}{|c|c|} \hline
 \text{Knot} & 3_1 \\ \hline
 \text{Kh} & q^9 t^3+q^5 t^2+q^3+q \\ \hline
 v_0 & t^3 x^9+t^2 x^5+x^3+x \\ \hline
 v_1 & 9 t^3 x^9+5 t^2 x^5+3 x^3+x \\ \hline
 v_2 & \frac{81 t^3 x^9}{2}+\frac{25 t^2 x^5}{2}+\frac{9 x^3}{2}+\frac{x}{2}
   \\ \hline
 v_3 & \frac{243 t^3 x^9}{2}+\frac{125 t^2 x^5}{6}+\frac{9 x^3}{2}+\frac{x}{6}
   \\ \hline
 v_4 & \frac{2187 t^3 x^9}{8}+\frac{625 t^2 x^5}{24}+\frac{27
   x^3}{8}+\frac{x}{24} \\ \hline
 v_5 & \frac{19683 t^3 x^9}{40}+\frac{625 t^2 x^5}{24}+\frac{81
   x^3}{40}+\frac{x}{120} \\ \hline
\end{array}
\end{align*}

\begin{align*}
 \begin{array}{|c|c|} \hline
 \text{Knot} & 4_1 \\ \hline
 \text{Kh} & q^5 t^2+\frac{1}{q^5 t^2}+q t+\frac{1}{q t}+q+\frac{1}{q} \\ \hline
 v_0 & t^2 x^5+\frac{1}{t^2 x^5}+t x+\frac{1}{t x}+x+\frac{1}{x} \\ \hline
 v_1 & 5 t^2 x^5-\frac{5}{t^2 x^5}+t x-\frac{1}{t x}+x-\frac{1}{x} \\ \hline
 v_2 & \frac{25 t^2 x^5}{2}+\frac{25}{2 t^2 x^5}+\frac{t x}{2}+\frac{1}{2 t
   x}+\frac{x}{2}+\frac{1}{2 x} \\ \hline
 v_3 & \frac{125 t^2 x^5}{6}-\frac{125}{6 t^2 x^5}+\frac{t x}{6}-\frac{1}{6
   t x}+\frac{x}{6}-\frac{1}{6 x} \\ \hline
 v_4 & \frac{625 t^2 x^5}{24}+\frac{625}{24 t^2 x^5}+\frac{t
   x}{24}+\frac{1}{24 t x}+\frac{x}{24}+\frac{1}{24 x} \\ \hline
 v_5 & \frac{625 t^2 x^5}{24}-\frac{625}{24 t^2 x^5}+\frac{t
   x}{120}-\frac{1}{120 t x}+\frac{x}{120}-\frac{1}{120 x} \\ \hline
\end{array}
\end{align*}

\begin{align*}
 \begin{array}{|c|c|}\hline
 \text{Knot} & 5_1 \\ \hline
 \text{Kh} & q^{15} t^5+q^{11} t^4+q^{11} t^3+q^7 t^2+q^5+q^3 \\ \hline
 v_0 & t^5 x^{15}+t^4 x^{11}+t^3 x^{11}+t^2 x^7+x^5+x^3 \\ \hline
 v_1 & 15 t^5 x^{15}+11 t^4 x^{11}+11 t^3 x^{11}+7 t^2 x^7+5 x^5+3 x^3 \\ \hline
 v_2 & \frac{225 t^5 x^{15}}{2}+\frac{121 t^4 x^{11}}{2}+\frac{121 t^3
   x^{11}}{2}+\frac{49 t^2 x^7}{2}+\frac{25 x^5}{2}+\frac{9 x^3}{2} \\ \hline
 v_3 & \frac{1125 t^5 x^{15}}{2}+\frac{1331 t^4 x^{11}}{6}+\frac{1331 t^3
   x^{11}}{6}+\frac{343 t^2 x^7}{6}+\frac{125 x^5}{6}+\frac{9 x^3}{2} \\ \hline
 v_4 & \frac{16875 t^5 x^{15}}{8}+\frac{14641 t^4 x^{11}}{24}+\frac{14641
   t^3 x^{11}}{24}+\frac{2401 t^2 x^7}{24}+\frac{625 x^5}{24}+\frac{27
   x^3}{8} \\ \hline
 v_5 & \frac{50625 t^5 x^{15}}{8}+\frac{161051 t^4 x^{11}}{120}+\frac{161051
   t^3 x^{11}}{120}+\frac{16807 t^2 x^7}{120}+\frac{625 x^5}{24}+\frac{81
   x^3}{40} \\ \hline
\end{array}
\end{align*}

\begin{align*}
 \begin{array}{|c|c|}\hline
 \text{Knot} & 5_2 \\ \hline
 \text{Kh} & q^{13} t^5+q^9 t^4+q^9 t^3+q^7 t^2+q^5 t^2+q^3 t+q^3+q \\ \hline
 v_0 & t^5 x^{13}+t^4 x^9+t^3 x^9+t^2 x^7+t^2 x^5+t x^3+x^3+x \\ \hline
 v_1 & 13 t^5 x^{13}+9 t^4 x^9+9 t^3 x^9+7 t^2 x^7+5 t^2 x^5+3 t x^3+3 x^3+x
   \\ \hline
 v_2 & \frac{169 t^5 x^{13}}{2}+\frac{81 t^4 x^9}{2}+\frac{81 t^3
   x^9}{2}+\frac{49 t^2 x^7}{2}+\frac{25 t^2 x^5}{2}+\frac{9 t
   x^3}{2}+\frac{9 x^3}{2}+\frac{x}{2} \\ \hline
 v_3 & \frac{2197 t^5 x^{13}}{6}+\frac{243 t^4 x^9}{2}+\frac{243 t^3
   x^9}{2}+\frac{343 t^2 x^7}{6}+\frac{125 t^2 x^5}{6}+\frac{9 t
   x^3}{2}+\frac{9 x^3}{2}+\frac{x}{6} \\ \hline
 v_4 & \frac{28561 t^5 x^{13}}{24}+\frac{2187 t^4 x^9}{8}+\frac{2187 t^3
   x^9}{8}+\frac{2401 t^2 x^7}{24}+\frac{625 t^2 x^5}{24}+\frac{27 t
   x^3}{8}+\frac{27 x^3}{8}+\frac{x}{24} \\ \hline
 v_5 & \frac{371293 t^5 x^{13}}{120}+\frac{19683 t^4 x^9}{40}+\frac{19683
   t^3 x^9}{40}+\frac{16807 t^2 x^7}{120}+\frac{625 t^2 x^5}{24}+\frac{81 t
   x^3}{40}+\frac{81 x^3}{40}+\frac{x}{120} \\ \hline
\end{array}
\end{align*}

\begin{align*}
 \begin{array}{|c|c|}\hline
 \text{Knot} & 6_1 \\ \hline
 \text{Kh} & q^9 t^4+q^5 t^3+q^5 t^2+\frac{1}{q^5 t^2}+q^3 t+q t+\frac{1}{q
   t}+q+\frac{2}{q} \\ \hline
 v_0 & t^4 x^9+t^3 x^5+t^2 x^5+\frac{1}{t^2 x^5}+t x^3+t x+\frac{1}{t
   x}+x+\frac{2}{x} \\ \hline
 v_1 & 9 t^4 x^9+5 t^3 x^5+5 t^2 x^5-\frac{5}{t^2 x^5}+3 t x^3+t
   x-\frac{1}{t x}+x-\frac{2}{x} \\ \hline
 v_2 & \frac{81 t^4 x^9}{2}+\frac{25 t^3 x^5}{2}+\frac{25 t^2
   x^5}{2}+\frac{25}{2 t^2 x^5}+\frac{9 t x^3}{2}+\frac{t x}{2}+\frac{1}{2 t
   x}+\frac{x}{2}+\frac{1}{x} \\ \hline
 v_3 & \frac{243 t^4 x^9}{2}+\frac{125 t^3 x^5}{6}+\frac{125 t^2
   x^5}{6}-\frac{125}{6 t^2 x^5}+\frac{9 t x^3}{2}+\frac{t x}{6}-\frac{1}{6
   t x}+\frac{x}{6}-\frac{1}{3 x} \\ \hline
 v_4 & \frac{2187 t^4 x^9}{8}+\frac{625 t^3 x^5}{24}+\frac{625 t^2
   x^5}{24}+\frac{625}{24 t^2 x^5}+\frac{27 t x^3}{8}+\frac{t
   x}{24}+\frac{1}{24 t x}+\frac{x}{24}+\frac{1}{12 x} \\ \hline
 v_5 & \frac{19683 t^4 x^9}{40}+\frac{625 t^3 x^5}{24}+\frac{625 t^2
   x^5}{24}-\frac{625}{24 t^2 x^5}+\frac{81 t x^3}{40}+\frac{t
   x}{120}-\frac{1}{120 t x}+\frac{x}{120}-\frac{1}{60 x} \\ \hline
\end{array}
\end{align*}

\begin{align*}
 \begin{array}{|c|c|}\hline
 \text{Knot} & 6_2 \\ \hline
 \text{Kh} & q^{11} t^4+q^9 t^3+q^7 t^3+q^7 t^2+q^5 t^2+q^5 t+\frac{1}{q^3
   t^2}+q^3 t+q^3+\frac{q}{t}+2 q \\ \hline
 v_0 & t^4 x^{11}+t^3 x^9+t^3 x^7+t^2 x^7+t^2 x^5+\frac{1}{t^2 x^3}+t x^5+t
   x^3+\frac{x}{t}+x^3+2 x \\ \hline
 v_1 & 11 t^4 x^{11}+9 t^3 x^9+7 t^3 x^7+7 t^2 x^7+5 t^2 x^5-\frac{3}{t^2
   x^3}+5 t x^5+3 t x^3+\frac{x}{t}+3 x^3+2 x \\ \hline
 v_2 & \frac{121 t^4 x^{11}}{2}+\frac{81 t^3 x^9}{2}+\frac{49 t^3
   x^7}{2}+\frac{49 t^2 x^7}{2}+\frac{25 t^2 x^5}{2}+\frac{9}{2 t^2
   x^3}+\frac{25 t x^5}{2}+\frac{9 t x^3}{2}+\frac{x}{2 t}+\frac{9 x^3}{2}+x
   \\ \hline
 v_3 & \frac{1331 t^4 x^{11}}{6}+\frac{243 t^3 x^9}{2}+\frac{343 t^3
   x^7}{6}+\frac{343 t^2 x^7}{6}+\frac{125 t^2 x^5}{6}-\frac{9}{2 t^2
   x^3}+\frac{125 t x^5}{6}+\frac{9 t x^3}{2}+\frac{x}{6 t}+\frac{9
   x^3}{2}+\frac{x}{3} \\ \hline
 v_4 & \frac{14641 t^4 x^{11}}{24}+\frac{2187 t^3 x^9}{8}+\frac{2401 t^3
   x^7}{24}+\frac{2401 t^2 x^7}{24}+\frac{625 t^2 x^5}{24}+\frac{27}{8 t^2
   x^3}+\frac{625 t x^5}{24}+\frac{27 t x^3}{8}+\frac{x}{24 t}+\frac{27
   x^3}{8}+\frac{x}{12} \\ \hline
 v_5 & \frac{161051 t^4 x^{11}}{120}+\frac{19683 t^3 x^9}{40}+\frac{16807
   t^3 x^7}{120}+\frac{16807 t^2 x^7}{120}+\frac{625 t^2
   x^5}{24}-\frac{81}{40 t^2 x^3}+\frac{625 t x^5}{24}+\frac{81 t
   x^3}{40}+\frac{x}{120 t}+\frac{81 x^3}{40}+\frac{x}{60} \\ \hline
\end{array}
\end{align*}

\begin{align*}
 \begin{array}{|c|c|}\hline
 \text{Knot} & 6_3 \\ \hline
 \text{Kh} & q^7 t^3+\frac{1}{q^7 t^3}+q^5 t^2+\frac{1}{q^5 t^2}+q^3
   t^2+\frac{1}{q^3 t^2}+q^3 t+\frac{1}{q^3 t}+q t+\frac{1}{q t}+2
   q+\frac{2}{q} \\ \hline
 v_0 & t^3 x^7+\frac{1}{t^3 x^7}+t^2 x^5+\frac{1}{t^2 x^5}+t^2
   x^3+\frac{1}{t^2 x^3}+t x^3+\frac{1}{t x^3}+t x+\frac{1}{t x}+2
   x+\frac{2}{x} \\ \hline
 v_1 & 7 t^3 x^7-\frac{7}{t^3 x^7}+5 t^2 x^5-\frac{5}{t^2 x^5}+3 t^2
   x^3-\frac{3}{t^2 x^3}+3 t x^3-\frac{3}{t x^3}+t x-\frac{1}{t x}+2
   x-\frac{2}{x} \\ \hline
 v_2 & \frac{49 t^3 x^7}{2}+\frac{49}{2 t^3 x^7}+\frac{25 t^2
   x^5}{2}+\frac{25}{2 t^2 x^5}+\frac{9 t^2 x^3}{2}+\frac{9}{2 t^2
   x^3}+\frac{9 t x^3}{2}+\frac{9}{2 t x^3}+\frac{t x}{2}+\frac{1}{2 t
   x}+x+\frac{1}{x} \\ \hline
 v_3 & \frac{343 t^3 x^7}{6}-\frac{343}{6 t^3 x^7}+\frac{125 t^2
   x^5}{6}-\frac{125}{6 t^2 x^5}+\frac{9 t^2 x^3}{2}-\frac{9}{2 t^2
   x^3}+\frac{9 t x^3}{2}-\frac{9}{2 t x^3}+\frac{t x}{6}-\frac{1}{6 t
   x}+\frac{x}{3}-\frac{1}{3 x} \\ \hline
 v_4 & \frac{2401 t^3 x^7}{24}+\frac{2401}{24 t^3 x^7}+\frac{625 t^2
   x^5}{24}+\frac{625}{24 t^2 x^5}+\frac{27 t^2 x^3}{8}+\frac{27}{8 t^2
   x^3}+\frac{27 t x^3}{8}+\frac{27}{8 t x^3}+\frac{t x}{24}+\frac{1}{24 t
   x}+\frac{x}{12}+\frac{1}{12 x} \\ \hline
 v_5 & \frac{16807 t^3 x^7}{120}-\frac{16807}{120 t^3 x^7}+\frac{625 t^2
   x^5}{24}-\frac{625}{24 t^2 x^5}+\frac{81 t^2 x^3}{40}-\frac{81}{40 t^2
   x^3}+\frac{81 t x^3}{40}-\frac{81}{40 t x^3}+\frac{t x}{120}-\frac{1}{120
   t x}+\frac{x}{60}-\frac{1}{60 x} \\ \hline
\end{array}
\end{align*}


\end{document}